\documentclass{amsart}
\usepackage{color}

\newtheorem{theorem}{Theorem}[section]
\newtheorem{lemma}[theorem]{Lemma}

\theoremstyle{definition}

\newtheorem{proposition}[theorem]{Proposition}
\theoremstyle{remark}
\newtheorem{remark}[theorem]{Remark}

\numberwithin{equation}{section}

\begin{document}
 
\title{Infinitely many segregated vector solutions of Schrodinger system}

\author{Ohsang Kwon}
\address{Department of Mathematics, Chungbuk National University, Cheongju, South Korea}
\email{ohsangkwon@chungbuk.ac.kr}
\thanks{O. Kwon is supported by }

\author{Min-Gi Lee}
\address{Department of Mathematics, Kyungpook National University, Daegu, South Korea}
\email{leem@knu.ac.kr}
\thanks{M.-G Lee is supported by }

\author{Youngae Lee}
\address{Department of Mathematical Sciences, College of Natural Sciences,  Ulsan National Institute of Science and Technology (UNIST), South Korea}
\email{youngaelee@unist.ac.kr}
\thanks{ Y. Lee is supported by the National Research Foundation of Korea(NRF) grant funded by the Korea government(MSIT) (No. NRF-2018R1C1B6003403). }


\date{\today}



\begin{abstract}
We consider the following system of Schr\"odinger equations
\begin{equation*}\left.\begin{cases}
 -\Delta U + \lambda U = \alpha_0 U^3+ \beta UV^2\\
 -\Delta V + \mu(y) V = \alpha_1 V^3+\beta U^2V 
\end{cases}\right. \text{in} \quad \mathbb{R}^N, \ N=2,  3,\end{equation*}
where  $\lambda$, $\alpha_0$, $\alpha_1>0$ are positive constants,    $\beta \in \mathbb{R}$ is the  coupling constant, and $\mu: \mathbb{R}^N \rightarrow \mathbb{R}$ is a potential function. Continuing the work of Lin and Peng  \cite{lin_peng_2014}, we present a solution of the type where one species has a peak at the origin and the other species has many peaks over a circle, but as seen in the above, coupling terms are nonlinear. 
\end{abstract}

\maketitle



\section{Introduction} \label{sec:intro}

We consider the following system of Schr\"odinger equations
\begin{equation}\left.\begin{cases}\label{system}
 -\Delta U + \lambda U = \alpha_0 U^3+ \beta UV^2\\
 -\Delta V + \mu(y) V = \alpha_1 V^3+\beta U^2V 
\end{cases}\right. \text{in} \quad \mathbb{R}^N, \ N=2,  3,\end{equation}
where  $\lambda$, $\alpha_0$, $\alpha_1>0$ are positive constants,    $\beta \in \mathbb{R}$ is the  coupling constant, and $\mu: \mathbb{R}^N \rightarrow \mathbb{R}$ is a potential function.  Our objective is to construct infinitely many solutions of \eqref{system} where the radial symmetry is broken to exhibit many peaks over a circle and at the origin.

This study answers a question among many problems from exploring the radial symmetry, or breaking it, of the solutions of nonlinear Schr\"odinger equations. By the work of  Gidas et. al. \cite{GNN} via the moving plane method, the following radial symmetry result is well-known, i.e., for the scalar Schr\"odinger equation
\begin{equation} \label{single}
 -\Delta U+ \mu(|y|) U = |U|^{p-1}U \quad \text{in $ \mathbb{R}^N$} 
\end{equation}
if the potential function $\mu$ is radial and non-decreasing in the radial variable, then the solution of \eqref{single} must be radial. Thus, the solutions with interesting patterns breaking the radial symmetry, such as solutions with many peaks, can only exist after violating the assumptions of this problem. Further, it is an interesting question if the cross interactions between species, where one considers the system of coupled Schr\"odinger equations, provide different circumstances for the existence of non-radial solutions.

For the scalar Schr\"odinger equation, by omitting the non-decreasing condition, i.e., $\mu$ is merely radial, Wei and Yan \cite{wei_yan_2014} presented such an example. They considered a $\mu$ that possesses the prescribed asymptotic behavior
\begin{align} \label{asymp0}
 \mu(|y|) = \mu_0 + \frac{a}{|y|^m} + {O}\left(\frac{1}{|y|^{m+\theta}}\right) \quad \text{for some $\mu_0,a,\theta>0$ and $m>1$ as $|y| \rightarrow \infty$}.
\end{align}
The constructed solution $u$ has segregated $k$ peaks over a circle. More specifically, $u$ peaks about $k$ points $x_i=(z_i,\mathbf{0})$, $i=1,\cdots,k$ where
$$z_i= \left( R \cos\left(\frac{2(i-1)\pi}{k}\right), ~R \sin\left(\frac{2(i-1)\pi}{k}\right)\right)\in \mathbb{R}^2 \quad \text{for $i=1,\cdots,k$}$$
and the radius $R$ is sufficiently large, which dependes on the number of peaks $k$. They showed that there are infinitely many such solutions for every $k$ larger than a certain $k_0$. Cerami, et al. \cite{cerami_passaseo_solimini_2015} showed that such a result holds when the radial symmetry of $\mu$ is also broken. 

Also, it has been actively studied if radial symmetry result of the scalar case continues for the problems of coupled Schr\"odinger systems. If $n$-species interact, we employ $n$ potential functions. We adopt a naive classification by behaviors of potentials. We will speak of a potential that is a constant function, an important case among non-decreasing potentials, or a potential that is an increasing function, and so on. In regard with this, for the prescribed asymptotic behaviors similar to \eqref{asymp0},  asymptotically, $a\le 0$ and  $a>0$ correspond to the non-decreasing and decreasing potentials, respectively.

Having those said, we remark on a few relevant works. Peng and Wang \cite{peng_wang_2013} considered two species Schr\"odinger system 
\begin{equation*}\begin{cases}
 -\Delta U+ P(|x|)U = \alpha_0 U^3 + \beta V^2 U,\\
 -\Delta V+ Q(|x|)V = \beta V^3 + \alpha_0 U^2 V,
\end{cases}\end{equation*}
and constructed various interesting solutions. They generated solutions of patterns where two species peak in a synchronized manner at shared sites or in a segregated manner at respective sites over a circle. The assumption on the potential function were such that one of $P(\cdot)$ and $Q(\cdot)$ violates the non-decreasing condition, or some other sophisticated conditions were supposed. See Theorem 1.1 and 1.2 in \cite{peng_wang_2013}. Wang et al. \cite{wang_zhao_2017} has improved the results of \cite{peng_wang_2013}, allowing larger class of potentials. The second leading order in \eqref{asymp0} with $m>0$ has been included. Long, et. al. \cite{long_tang_yang_2020} considered a problem where two potentials are  exponentially decreasing and obtained the result. The problems with severer violations, further omitting the radial symmetry of potentials, have been investigated by Ao and Wei \cite{ao_wei_2014} and Ao, et. al. \cite{ao_wang_yao_2016}. The case with nonlinearity other than $p=3$ have been studied by Zheng \cite{zheng_2017}. Recently, Peng, et. al. \cite{peng_wang_wang_2019} constructed non-radial solutions in the three species problem, where the potentials are all positive constant; thus non-decreasing. Their results clearly show that it is possible to have non-radial patterns in the system case even with radial and non-increasing potentials. 
The results can be compared to those of Wei and Yan \cite{wei_yan_2014}. Not infinitely many solutions are constructed and results follow under a certain condition on the constants.

Lin and Peng \cite{lin_peng_2014} advanced a question in another direction. They sought a different pattern. In their solution, while one species peaks at many points over a circle as before, the other species peaks at the origin. They constructed solutions of these patterns in a model where the coupling terms are linear.

Our work in this paper can be summarized as follows. Continuing the work of Lin and Peng  \cite{lin_peng_2014}, we seek a solution of the type where one species has a peak at the origin but we work with the two species model \eqref{system} where the coupling terms are nonlinear. We also consider one potential to be a constant function, while the other potential $\mu(|y|)$ violates the non-decreasing condition. We though mention that we have made a small improvement from the work of Wei and Yan \cite{wei_yan_2014} and others that the asymptotic behavior of $\mu$ is assumed to satisfy the following with $m> \frac{1}{2}$:
\begin{equation}\tag{$A$}\label{A}
 \begin{aligned}
  &1. \quad \mu(y) = \mu_0 + \frac{a}{|y|^m} + O\left(\frac{1}{|y|^{m+\theta}}\right) \quad \text{for some $\mu_0,a,\theta>0$ and $m> \frac{1}{2}$ as $|y| \rightarrow \infty$.}\\
  &2. \quad \text{$\mu$ is bounded and $\mu(y)\ge \tilde{\mu}_0 > 0$ for all $y$.}
 \end{aligned}
\end{equation}
Now we state our main theorem. $U_0$, $(U_i)_{i=1}^k$, the exponent $p$, and the fraction number $f_0$ in the statement will be fixed in Section \ref{notions} and \ref{sec:result}.
\begin{theorem} \label{mainthm} Suppose \eqref{A} is satisfied with $m > \frac{1}{2}$ and $\left|\frac{\beta}{\mu_0}\right| < f_0$. Then $\exists k_0$ such that for $k\ge k_0$ there exists a solution of the form $\Big(U_0 +u, \displaystyle \sum_{i=1}^k V_i + v\Big)$ of \eqref{system} where for some constant $C>0$ and $p> 0$ $ \|u\|_{H^1},  \|v\|_{H^1} \le Ck^{-p}$.
\end{theorem}
\begin{remark}
  Using dilation arguments, it is sufficient to give a proof of Theorem \ref{mainthm} under the assumption $\mu_0=1$ in \eqref{A} for the  case $|\beta| \le f_0$. For the rest of the paper, $\mu_0=1$.
\end{remark}
The remainder of the paper is organized as follows. In Section \ref{notions}, we introduce some notions. We prove the main theorem in the Section \ref{sec:result}.

\section{Notations} \label{notions}

First, we let $U_{0}$ and $V_{0}$ be respectively the unique  {positive} radial solutions (see Kwong \cite{Kwong}) of
\begin{equation}\label{u0v0}
 -\Delta U_{0} + \lambda U_{0} - \alpha_0 U_{0}^3 = 0, \quad \text{and} \quad -\Delta V_{0} + V_{0} - \alpha_1 V_{0}^3 = 0,
\end{equation}
Throughout this paper, $M$ will denote the constant so that
\begin{equation}\label{exp}U_{0}(y) \le Me^{-\sqrt{\lambda}|y|}~{\rm min}\left\{1, |y|^{-\left(\frac{N-1}{2}\right)}\right\}, \quad \text{and} \quad V_{0}(y) \le Me^{-|y|}~{\rm min}\left\{1, |y|^{-\left(\frac{N-1}{2}\right)}\right\}.\end{equation}

We fix an integer $k>1$, which will be the number of bumps away from the origin for the second component in \ref{system}, and let \[V_{i}(y):=V_{0}(y-x_i)\quad\mbox{for}\  i=1,\cdots,k.\] For the points
$$z_i= \left( R \cos\left(\frac{2(i-1)\pi}{k}\right), ~R \sin\left(\frac{2(i-1)\pi}{k}\right)\right) \quad \text{for $i=1,\cdots,k$,}$$
$x_i = z_i$ for $N=2$, and $x_i = (z_i,0)$ for $N=3$.  Here, $R>0$ will be chosen sufficiently large later, which depends on $k$. It is also convenient to subdivide the plane $\mathbb{R}^2$ into $k$ sectors
\begin{equation}\label{sectors}
\Omega^\prime_i := \left\{z \in \mathbb{R}^2~\Big|~ z_i\cdot z \ge  {R}|z|\cos\left(\frac{\pi}{k}\right)\right\} \quad i=1,\cdots,k 
\end{equation}
and we set $\Omega_i = \Omega^\prime_i$ if $N=2$ and $\Omega_i = \Omega^\prime_i \times \mathbb{R}$ if $N=3$.

On $H^1( \mathbb{R}^N)$, we introduce two inner products
$$ \left<u,v\right>_0:= \int_{ \mathbb{R}^N} \nabla u\cdot\nabla v + \lambda uv \quad \text{and} \quad \left<u,v\right>_1:= \int_{ \mathbb{R}^N} \nabla u\cdot\nabla v + \mu(y) uv.$$ and two norms $\|\cdot\|_0:=\sqrt{\left<\cdot,\cdot\right>_0}$, and $\|\cdot\|_1:=\sqrt{\left<\cdot,\cdot\right>_1}$. We fix a closed subspace $H_s$ of $H^1( \mathbb{R}^N)$ possessing the following symmetry. Let $Q$ be a couter-clockwise rotating matrix of an angle $\frac{2\pi}{k}$ for the first two coordinates. We define 
\begin{equation*}
 H_s:= \left\{ u \in H^1 (\mathbb{R}^N) ~\Big|~ \text{$u(Q^iy) = u(y)$ for $i=1,\cdots,k$ and $u$ is even in $y_n$, $n=2,\cdots,N$}\right\}.
\end{equation*}
and the closed subspace $\hat{H}_s$
\begin{equation}
 \begin{aligned} 
    \hat{H}_s:= \left\{ v \in H_s ~\Big|~  \sum_{i=1}^k\int_{\mathbb{R}^N}  V_{i}^2 \frac{\partial V_{i}}{\partial R} \,v~dy = 0 \right\}.
 \end{aligned}
\end{equation}
We define $E:= H_s \times \hat{H}_s$ and we equip $E$ with the norm \[\|(u,v)\|:= \mbox{max} \left\{ \| u\|_0, \| v\|_1\right\}.\] 

Nextly, since $m> \frac{1}{2}$ there exists $\tau_0$ such that 
$$\tau_0 m - \frac{1}{2} > 0 \quad \text{and} \quad \frac{1}{2} < \tau_0 < 1.$$ We set 
\begin{equation} \label{delta0}
4\delta_0:= {\rm min} \left\{ \tau_0m - \frac{1}{2}, \tau_0 - \frac{1}{2}, \theta \right\}>0, \quad  p:= \tau_0m - \frac{1}{2} - \delta_0 > 0
\end{equation} 
and define the inverval for each $k$
\begin{equation}\label{sk}{S_k:=\left[\left(\frac{m - \delta_0}{2\pi} \right)k\log k,\left(\frac{m+\delta_0}{2\pi}\right)k\log k\right]}.
\end{equation}
We will only consider $R$ that is in the inverval $S_k$ for each $k$. This enables us to take $k$ as our main parameter, and the notations ${O(\cdot)}$ and $o(\cdot)$ are always as $k \rightarrow \infty$.

\section{Results} \label{sec:result}
Our solution scheme is based on the following formal observations. Suppose $\Big(U_0 +u, \displaystyle \sum_{i=1}^k V_i + v\Big)$ is a solution of \eqref{system}, or $(u,v)$ formally solves
\begin{equation*}
 \left\{
 \begin{aligned}
 &-\Delta u + \lambda u - 3\alpha_0U_0^2 u =g_0(u,v),  \\
 &-\Delta v + \mu(y) v - 3\alpha_1\left(\sum_{i=1}^k V_{i}\right)^2 v =g_1(u,v),  
 \end{aligned}
 \right.
\end{equation*}
where
\begin{equation*}
 \left\{
 \begin{aligned}
 & g_0(u,v):=3\alpha_0U_{ {0}} u^2 + \alpha_0u^3 + \beta(U_0+u)\left(\sum_{i=1}^k V_{i} +v\right)^2,\\
 & g_1(u,v):= 3\alpha_1\left(\sum_{i=1}^k V_{i}\right)v^2 + \alpha_1v^3 + \beta\left(U_0 +u\right)^2\left(\sum_{i=1}^k V_{i} +v\right)  \\
 & \quad \quad \quad \quad \quad \quad \quad - (\mu-1)\left(\sum_{i=1}^k V_{i}\right)+ \alpha_1\left\{ \left(\sum_{i=1}^k V_{i}\right)^3 - \sum_{i=1}^k V_{i}^3 \right\}.
 \end{aligned}
 \right.
\end{equation*}
Based on these observations, for a fixed $u \in H_s$, we consider the following linear functional $\ell_u$ on $H_s$:
$$\ell_u[\varphi] := \int_{ \mathbb{R}^N}\left( \nabla u\cdot\nabla\varphi + \lambda u \varphi - 3\alpha_0 U_0^2 u\varphi\right)dy,$$
which is obviously bounded. This in turn, via Riesz representation theorem, defines the linear operator $L_0: H_s\rightarrow H_s$ by the defining relation
\begin{equation}\label{l0}\left<L_0(u),\varphi\right>_0 = \ell_u[\varphi].\end{equation}
For a fixed $v \in \hat{H}_s$,
$$\tilde\ell_v[\phi]:=\int_{ \mathbb{R}^N}\left( \nabla v\cdot\nabla\phi + \mu(y) v \phi - 3\alpha_1 \left(\sum_{i=1}^k V_{i}\right)^2 v\phi\right)dy,$$
and the linear operator $L_1: \hat{H}_s \rightarrow \hat{H}_s$ is defined by the relation
\begin{equation}\label{l1}\left<L_1(v),\phi\right>_1 = \tilde\ell_v[\phi].\end{equation}

On the other hand,  $g_0(u,v)$ and $g_1(u,v)$ {defines the linear functionals {$G_0$ and $G_1$,} respectively on $H_s$ and $\hat{H}_s$} 
\begin{align*}
 G_0(u,v)[\varphi]:=\int_{ \mathbb{R}^N} g_0(u,v)\varphi \quad \text{for $\varphi \in H_s$}, \quad
 G_1(u,v)[\phi]:=\int_{ \mathbb{R}^N} g_1(u,v)\phi \quad \text{for $\phi \in \hat{H}_s$}.
\end{align*}
Bounded of these functionals are justified by the Sovolev embeddings for $N=2, 3$ and  calculations later in the proof of Proposition \ref{Fixed}.
Applying Riesz representation theorem, there exist $\Gamma_0(u,v) \in H_s$, $\Gamma_1(u,v) \in \hat{H}_s$ with
$$\left<\Gamma_0(u,v),\varphi\right>_0 = G_0(u,v)[\varphi] \quad \text{for $\varphi \in H_s$}, \quad \left<\Gamma_1(u,v),\phi\right>_1 = G_1(u,v)[\phi] \quad \text{for $\phi \in \hat{H}_s$}.$$
Combined with the inverses $L_0^{-1}$ and $L_1^{-1}$ on $H_s$ and $\hat{H}_s$ respectively (See Lemma \ref{inv0}), the map
$$(u,v) \mapsto (\bar{u},\bar{v})=\left(L_0^{-1}\Big(\Gamma_0(u,v)\Big), ~L_1^{-1}\Big(\Gamma_1(u,v)\Big)\right)$$
defines an operator $F: E \rightarrow E$. If $(u,v)$ is any fixed point of $F$ then it holds that
\begin{align*}
 &\left(L_0(u),L_1(v)\right) = \left(\Gamma_0(u,v),\Gamma_1(u,v)\right)\end{align*}
 if and only if 
 \begin{align*}
 \left\{ 
 \begin{aligned}
  &\int_{\mathbb{R}^N} \nabla u \cdot \nabla \varphi + \lambda u\varphi - 3\alpha_0U_0^2 u\varphi = \int_{\mathbb{R}^N} g_0(u,v)\varphi \quad\mbox{for all}\  \varphi\in H_s,\\
  &\int_{\mathbb{R}^N} \nabla v \cdot \nabla \phi + \mu(y) v\phi - 3\alpha_1\left(\sum_{i=1}^k V_{i}\right)^2v\phi = \int_{\mathbb{R}^N} g_1(u,v)\phi \quad \mbox{for all}\ \phi \in \hat{H}_s.
 \end{aligned}
 \right.
\end{align*}

The following results of the invertibilities of $L_0$ and $L_1$  are known.
\begin{lemma} \cite{Kwong},  \cite[Lemma 2.1]{wei_yan_2014}  \label{inv0} There are constants $\rho_0, \rho_1>0$ satisfying
\begin{align*}
\rho_0 \|\varphi\|_0 &\le  \|L_0 \varphi\|_0  \le (\rho_0)^{-1}\|\varphi\|_0  \quad \text{for all $\varphi \in H_s$},\ \mbox{and} \\
\rho_1 \|\phi\|_1 &\le  \|L_1 \phi\|_1 \le (\rho_1)^{-1}\|\phi\|_1 \quad \text{for all $\phi \in \hat{H}_s$}
\end{align*}
\end{lemma}

Nextly, we  establish the uniform bound of the sum $\displaystyle \sum_{i=1}^k V_{i}$ independent of the number of bumps $k$, provided that $k$ and $\frac{R}{k}$ are both sufficiently large. We provide the proof of a tweaked version (allowing $\eta \in (0,2])$ of Lemma A.1 in \cite{wei_yan_2014}.

\begin{lemma}\label{ksum}  \cite[Lemma A.1]{wei_yan_2014} Fix any $\eta \in (0,2]$. For some $k_0>0$, for any $k\ge k_0$, for any $R \in S_k$, and for $y \in{\Omega}_1$
\begin{equation} \label{ksumest} \displaystyle \sum_{i=2}^k V_{i}(y) \le 6Me^{-\frac{\eta R\pi }{k}} e^{(-1+\eta)|y-x_1|}\quad \text{and} \quad \displaystyle \sum_{i=1}^k V_{i}(y) \le 7Me^{(-1+\eta)|y-x_1|}.
\end{equation}

\end{lemma}
\begin{proof}
For each $i$, we have
\begin{align*}
 V_i(y) &\le M e^{-|y-x_i|} = Me^{ -\frac{\eta}{2}|x_i-x_1| }e^{ \frac{\eta}{2}|x_i-x_1| }e^{ -|y-x_i| }\\&
 \le Me^{ -\frac{\eta}{2}|x_i-x_1| }e^{ \frac{\eta}{2}|y-x_1| }e^{ ( \frac{\eta}{2}-1)|y-x_i| }
\le Me^{ -\frac{\eta}{2}|x_i-x_1| }e^{ ( -1 + \eta)|y-x_1| }, 
\end{align*}
where we have used that $ \frac{\eta}{2} -1 \le 0$ and that $|y-x_i|\ge |y-x_1|$ for $y\in \Omega_1$. Therefore
$$\sum_{i\ge 2} V_{i}(y) \le Me^{(-1+\eta)|y-x_1|}\sum_{i\ge 2} e^{-\frac{\eta}{2}|x_i-x_1|}.$$
We let $\theta_i = \frac{2\pi(i-1)}{k}$. Then $ \frac{|x_i-x_1|}{2} = R\sin\left(\frac{\theta_i}{2}\right)$ and we write
\begin{align}\label{sum2}
 \sum_{i\ge 2} e^{-\frac{\eta}{2}|x_i-x_1|} &\le 2\sum_{0<\theta_i \le \pi} e^{-\eta R\sin\left(\frac{\theta_i}{2}\right)} = 2 e^{-\eta R\sin\left(\frac{\theta_2}{2}\right)} + 2\sum_{\theta_2 <\theta_i \le \pi}e^{-\eta R\sin\left(\frac{\theta_i}{2}\right)} .
\end{align}
We may assume $k \ge 4$. Then $\sin\left(\frac{\theta_2}{2}\right) = \sin \left(\frac{\pi}{k}\right) = \frac{\pi}{k} - \frac{\sin(\xi_0)}{2}(\frac{\pi^2}{k^2})$ for some $\xi_0$ and
\begin{align*}
 e^{-\eta R\sin\left(\frac{\theta_2}{2}\right)} = e^{-\frac{\eta R \pi }{k}} \left(1 + O\left(\frac{R}{k^2}\right) \right)  = e^{-\frac{\eta \pi R}{k}} (1 + o(1)) \le 2 e^{-\frac{\eta  R\pi}{k}}
\end{align*}
for sufficiently large $k$. Now for the remaining sum, we use the fact that $ \sin\left(\frac{\theta_i}{2}\right) \ge \frac{\theta_i}{\pi}$ for $0\le\theta_i \le \pi$, and $e^{-\frac{2\eta R}{k}} \le \frac{1}{2}$ for $k$ large. Using $e^{-\frac{2\eta R}{k}}$ as the common ratio 
\begin{align}\label{sum3}
 2\sum_{{\theta_2 <\theta_i \le \pi}} e^{-\eta R\sin\left(\frac{\theta_i}{2}\right)} &\le 2\sum_{i\ge 3} e^{-\frac{\eta R}{\pi}\theta_i} \le 4 e^{-\frac{4\eta R}{k}} \le 4e^{-\frac{\eta R \pi }{k}}.
\end{align}
Assertions then straightforwardly follow.
 \end{proof}

Note that whenever \eqref{ksumest} holds with $\eta \in (0,1]$, we can define a number $\gamma_0$ with $\displaystyle \max\left\{U_0^2 , \left(\sum_{i=1}^k V_{i}\right)^2\right\} \le \gamma_0$ and a number $f_0$ with \begin{equation}\label{f0}0< f_0 < \frac{1}{\gamma_0}.\end{equation}
\begin{proposition} \label{Fixed} There exists $k_0$ such that for $k\ge k_0$ and $R \in S_k$, the map $F$ has a fixed point $(u_R,v_R) \in E$ and $\|(u_R,v_R)\| \le k^{-p}$.
 \end{proposition}
\begin{proof}
\textit{Step 1.} Let $B \subseteq E$ be the ball with $\|(u,v)\| \le k^{-p}$. We first show that   $F(B) \subseteq B$ for sufficiently large $k$  and for $R \in S_k$.

Let $(\bar u, \bar v) := F(u,v)$. Then,
 \begin{align*}
  \|\bar u\|_0 &= \left\|L_0^{-1}\Big(\Gamma_0(u,v)\Big)\right\|_0 \le \frac{1}{\rho_0}\left\|\Gamma_0(u,v)\right\|_0 = \frac{1}{\rho_0} \sup_{\|\varphi\|_0=1} \left| \int_{ \mathbb{R}^N} g_0(u,v) \varphi \right|\\
  &\le  {\frac{1}{\rho_0}} \sup_{\|\varphi\|_0=1} \Bigg[   {\int_{ \mathbb{R}^N} c_0}|\varphi| \left( |u|^2 +|v|^2 + |u^3| + |uv| + |uv^2|\right) + \left|  {2}\beta v U_0\left(\sum_{i=1}^k V_{i}\right)\varphi\right|dy\\
  & +  {\int_{ \mathbb{R}^N} }  \left|\beta U_0\left(\sum_{i=1}^k V_{i}\right)^2\varphi\right| +  \left|\beta u\left(\sum_{i=1}^k V_{i}\right)^2\varphi\right|dy\Bigg],
 \end{align*}
 where a constant $c_0>0$ depends only on $\alpha_0, \beta, \lambda$. Also,
 \begin{align*}
  \|\bar v\|_1 &= \left\|L_1^{-1}\Big(\Gamma_1(u,v)\Big)\right\|_1 \le \frac{1}{\rho_1}\left\|\Gamma_1(u,v)\right\|_1 = \frac{1}{\rho_1} \sup_{\|\phi\|_1=1} \left| \int_{ \mathbb{R}^N} g_1(u,v) \phi \right|\\
  &\le  {\frac{1}{\rho_1}}\sup_{\|\phi\|_1=1} \Bigg[  {\int_{ \mathbb{R}^N} c_1} |\phi|\left( |u^2| + |v^2|+ |v^3| + |uv| +|u^2v|\right)  \\
  &+ \int_{\mathbb{R}^N} \left| {2} \beta u U_0\left(\sum_{i=1}^k V_{i}\right)\phi\right| + \left|\beta U_0^2\left(\sum_{i=1}^k V_{i}\right)\phi\right| + \left|(\mu(y)-1)\left(\sum_{i=1}^k V_{i}\right)\phi\right|dy\\
  &  + \int_{\mathbb{R}^N}\left|\left(\sum_{i=1}^k V_{i}\right)^3 - \sum_{i=1}^k V_{i}^3 \right| |\phi| + \left|\beta v U_0^2  \phi\right| \Bigg]dy,
 \end{align*}
 where a constant $c_1$ depends only on $\alpha_1, \beta$.  
We set
\begin{align*}
 I_1 &= \beta \int_{\mathbb{R}^N} \left(\sum_{i=1}^k V_{i}\right)^2 u\varphi + U_0^2 v\phi \:dy + 2\beta\int_{\mathbb{R}^N} \left(U_0\sum_{i=1}^k V_i\right)(v\varphi + u\phi) \:dy, \\
 I_2 &= \beta\int_{\mathbb{R}^N} U_0\left(\sum_{i=1}^k V_{i}\right)^2\varphi + U_0^2\left(\sum_{i=1}^k V_{i}\right)\phi \: dy,\\
 I_3 &= \int_{\mathbb{R}^N}(\mu(y)-1)\left(\sum_{i=1}^k V_{i}\right)\phi\:dy + \int_{\mathbb{R}^N} \left(\left(\sum_{i=1}^k V_{i}\right)^3 - \sum_{i=1}^k V_{i}^3\right) \phi \:dy.
\end{align*}
For the estimation of $|I_1|$, we choose $\eta \in (0,1)$ then $y\in \Omega_1$ we have that
 \begin{equation} \label{crossprod}
  U_0\left(\sum_{i=1}^k V_{i}\right) \le C_0(M)e^{-\sqrt{\lambda}|y|}e^{-(1-\eta)|y-x_1|} \le C_1(M) e^{-\gamma R}
\end{equation}
 for some small constant $\gamma>0$ where $C_0(M), C_1(M)>0$ are constants independent of $k$. Moreover, by the symmetry we have $\left\|U_0\left(\sum_{i=1}^k V_{i}\right)\right\|_{L^\infty(\mathbb{R}^N)} \le C_1(M) e^{-\gamma R}$.  Hence
 \begin{align*}
  |I_1| \le \big(2|\beta| \gamma_0 + Ce^{-\gamma R}\big) \:\|(u,v)\|.
 \end{align*}

For the estimation of $|I_2|$, we see that 
 \begin{equation}\label{I2terms}
  \begin{aligned}
  {\int_{ \mathbb{R}^N} }\left | U_0\left(\sum_{i=1}^k V_{i}\right)^2\varphi  \right| dy&\le C\sqrt{k}\left\|U_0\left(\sum_{i=1}^k V_{i}\right)^2\right\|_{L^2(\Omega_1)}\|\varphi\|_{L^2( { \mathbb{R}^N})}  \\
   &\le C\sqrt{k}e^{-\gamma R}\left\|\sum_{i=1}^k V_{i}\right\|_{L^2(\Omega_1)}\|\varphi\|_{L^2( \mathbb{R}^N)}\\
 &\le C \sqrt{k} e^{-\gamma R}\|\varphi\|_{L^2( \mathbb{R}^N)}
 \end{aligned}
\end{equation}
 and using the similar argument we conclude that
\begin{align*}
|I_2| \le C \sqrt{k} e^{-\gamma R}.
\end{align*}

For the estimation of $|I_3|$, note that it is possible by \eqref{A} to choose large $k_0$ such that if $k\ge k_0$, then for $R \in S_k$, it holds that
$$ |\mu - 1| \le \frac{2a}{|y|^m} \quad \text{for $|y|\ge \frac{R}{2}$. } \quad \text{Then,}$$
\begin{equation} \label{mu-1}
\begin{aligned}
 &\int_{\mathbb{R}^N} \left|(\mu-1)\left(\sum_{i=1}^k V_{i}\right)\phi\right| dy \le \|\phi\|_{L^2(\mathbb{R}^N)} \sqrt{k}\left\| (\mu -1)\left(\sum_{i=1}^k V_{i}\right)\right\|_{L^2(\Omega_1)}\\
 &\le C_0\|\phi\|_{L^2(\mathbb{R}^N)} \sqrt{k}\left(\left\| (\mu -1)\left(\sum_{i=1}^k V_{i}\right)\right\|_{L^2(\Omega_1 \cap B_{\frac{R}{2}(x_1)})} +  \left\| (\mu -1)\left(\sum_{i=1}^k V_{i}\right)\right\|_{L^2(\Omega_1 \setminus  B_{\frac{R}{2}(x_1)})}\right)\\
 &\le C_1\|\phi\|_{L^2(\mathbb{R}^N)} \sqrt{k} \left\| (\mu -1)\right\|_{L^\infty( B_{\frac{R}{2}(x_1)} )} \left\|e^{-(1-\eta)|y-x_1|}\right\|_{L^2(\Omega_1)} + C\|\phi\|_{L^2(\mathbb{R}^N)} \sqrt{k} e^{-\gamma R}\\
 &\le C_2\|\phi\|_{L^2(\mathbb{R}^N)} \left( \frac{\sqrt{k}}{R^m} + \sqrt{k}e^{-\gamma R}\right),
\end{aligned}
\end{equation}where $C_0, C_1, C_2>0$ are constants independent of $k$.  Also,  
\begin{align*}
 \left(\sum_{i=1}^k V_{i}\right)^3 - \sum_{i=1}^k V_{i}^3 &= \left[ V_{1} \sum_{i\in \hat{1}} V_{i} + V_2 \sum_{i\in \hat{2}} V_{i} + \cdots + V_k \sum_{i\in \hat{k}} V_{i} \right] \left(\sum_{i=1}^k V_{i}\right) \\
 & \le (V_{1} + \sqrt{\gamma_0}) \left(\sum_{i\ge 2} V_{i}\right)\left(\sum_{i=1}^k V_{i}\right) \le 2\sqrt\gamma_0\left(\sum_{i\ge 2} V_{i}\right)\left(\sum_{i=1}^k V_{i}\right),\end{align*}
 where $\hat{\ell} = \{1,2,\cdots,\ell-1,\ell+1,\cdots,k\}$. Hence  
 \begin{equation}\begin{aligned}\label{sumdiff}
  \int \left(\left(\sum_{i=1}^k V_{i}\right)^3 - \sum_{i=1}^k V_{i}^3 \right)\phi &\le C\|\phi\|_{L^2(\mathbb{R}^N)} \sqrt{k}\left\|\left(\sum_{i\ge 2} V_{i}\right)\left(\sum_{i=1}^k V_{i}\right) \right\|_{L^2(\Omega_1)} \\
 &\le C\|\phi\|_{L^2(\mathbb{R}^N)}\sqrt{k} e^{-\frac{\eta'\pi R}{k}}\| e^{(-1+\eta)|y-x_1|}e^{(-1+\eta')|y-x_1|}\|_{L^2(\Omega_1)}\\
 & \le C \sqrt{k}e^{-\frac{\pi\eta' R}{k}},
 \end{aligned}\end{equation}
 where we have used the Lemma \ref{ksum} twice in such a way $\eta'\in (0,2)$, $\eta\in (0,1)$, and $-2 +\eta + \eta' < 0$. In particular we set $\eta'= 2\tau_0$. 
Therefore, we have that
\begin{equation} \label{p2}
 \|(\bar u,\bar v\| \le C\left( k^{-2p} + (\sqrt{k}+1)e^{-\gamma R} + \frac{\sqrt{k}}{R^m} + \sqrt{k}e^{-\frac{2\pi\tau_0 R}{k}} \right) + |\beta|\gamma_0k^{-p}.
\end{equation}
\eqref{delta0} and \eqref{sk} shows that
$$ \sqrt{k}e^{-\frac{2\pi\tau_0 R}{k}} \le k^{-\left(\tau_0m - \frac{1}{2}\right) + \delta_0 + (\tau_0-1)\delta_0} = k^{-p + (\tau_0-1)\delta_0}$$
Combining \eqref{p2} with the assumption $|\beta|< f_0<\frac{1}{\gamma_0}$, we have
that for $k$ large enough $$\|(\bar u,\bar v)\| \le k^{-p}.$$

\textit{Step 2.} Assertions then follow if we show that $F$ is a contraction in the ball $B$. For $(u_1,v_1)\in B$ and $(u_2,v_2) \in B$, let $(\bar{u}_1, \bar{v}_1) = F(u_1,v_1)$ and $(\bar{u}_2, \bar{v}_2) = F(u_2,v_2)$. Since we have 
\begin{align*}
&g_0(u_1,v_1) - g_0(u_2,v_2)\\
&= {(u_1 - u_2)}\left\{ 3\alpha_0U_0(u_1+u_2) + \alpha_0\big(u_1^2+u_1u_2 + u_2^2\big) + \beta \left(\sum_{i=1}^k V_{i}\right)^2 + 2\beta \left(\sum_{i=1}^k V_{i}\right) v_1 + \beta v_1^2\right\}\\
&+  {(v_1-v_2)}\beta\left\{ 2U_0 \left(\sum_{i=1}^k V_{i}\right) + U_0(v_1+v_2) + 2\left( \sum_{i=1}^k V_{i}\right)u_2 + u_2(v_1+v_2)\right\},\end{align*}and
\begin{align*}
&g_1(u_1,v_1) - g_2(u_2,v_2)\\
&= {(u_1 - u_2)}\beta\left\{ 2U_0\left(\sum_{i=1}^k V_{i}\right) + 2U_0v_1 + \left(\sum_{i=1}^k V_{i}\right)(u_1+u_2) +v_1(u_1+u_2)\right\}\\
&+  {(v_1-v_2)}\left\{ 3\alpha_1\left(\sum_{i=1}^k V_{i}\right)(v_1+v_2) + \alpha_1\big(v_1^2 + v_1v_2 + v_2^2\big) + \beta U_0^2 + 2\beta U_0u_2 + \beta u_2^2\right\},
\end{align*}
calculations in the first part of the proof show that
\begin{align*}
 \|(\bar{u}_1 - \bar{u}_2, \bar{v}_1 - \bar{v}_2)\| &\le \|(u_1 - u_2,v_1-v_2)\| \left( |\beta|\gamma_0 + Ck^{-p}\right). \\
\end{align*}
Therefore there exists $k_0$ so large that for $k \ge k_0$, $$\left( |\beta|\gamma_0 + Ck^{-p}\right) < 1.$$
\end{proof}

Having established the existence of $(u_R, v_R)$ for $R \in S_k$, we are able to advance the following procedure. We define
\begin{equation}
\begin{aligned}\label{energy}
I(U, V)=& \frac{1}{2} \int_{\mathbb{R}^{N}}\left(|\nabla U|^{2}+\lambda U^{2}+|\nabla V|^{2}+\mu(y) V^{2}\right) d y \\
&-\frac{1}{4} \int_{\mathbb{R}^{N}}\left(\alpha_0 U^{4}+\alpha_1 V^{4}+2 \beta U^{2} V^{2}\right) d y.
\end{aligned}
\end{equation}
then $I \in C^{2}\left(H_s \times \hat{H}_s\right)$. In view of Proposition \ref{Fixed}, there is a Lagrange multiplier $\Lambda_{R}$ satisfying
\begin{equation*} 
\left\{\begin{array}{l}
-\Delta U_{R}+\lambda U_{R}=\alpha_0 U_{R}^{3}+\beta V_{R}^{2} U_{R} \\
-\Delta V_{R}+V_{R}=\alpha_1 V_{R}^{3}+\beta U_{R}^{2} V_{R}+\Lambda_{R} \sum_{i=1}^{k} V_{i}^{2} \frac{\partial V_{i}}{\partial R}
\end{array}\right.
\end{equation*}
where $\left(U_{R}, V_{R}\right)=\left(U_{0}+u_{R}, \sum_{i=1}^{k} V_{i}+v_{R}\right)$. Let
\begin{equation}\label{def_F}
F(R)=I(U_R,V_R)=I\left(U_{0}+u_{R}, \sum_{i=1}^{k} V_{i}+v_{R}\right).\end{equation} Multiplying the second equation by $\frac{\partial V_{i}}{\partial R}$
and integrating in $\mathbb{R}^{N}$, we can see that $F^{\prime}(R)=0$ implies that the constant $\Lambda_{R}$ in $(9)$ vanishes, and thus $\left(U_{R}, V_{R}\right)$ is a solution to our main equation \eqref{system}. Above discussion shows that in order to complete the proof of Theorem \ref{mainthm}, it is enough to show that the maximization problem $\max _{R \in S_{k}} F(R)$ is achieved by an interior point of $S_{k}$.  

We write
\begin{equation}\begin{aligned}\label{FR}
F(R)&=I\left(U_{0}, \sum_{i=1}^{k} V_{i}\right)+l^R\left[u_{R}, v_{R}\right]+ q^R \left[u_{R}, v_{R}\right]+h^R\left[u_{R}, v_{R}\right]
\end{aligned}\end{equation}
where $l^R$ and $q^R$ are linear and quadratic form on $E$ respectively that are defined by
\begin{equation}
\begin{aligned}
l^R[u, v]:=& \int_{\mathbb{R}^{2}}(\mu(y)-1)\left(\sum_{i=1}^{k} V_{i}\right) v d y  +\int_{\mathbb{R}^{2}}\alpha_1\left(\sum_{i=1}^{k} V_{i}^{3}-\left(\sum_{i=1}^{k} V_{i}\right)^{3}\right) v d y \\
&-\beta \int_{\mathbb{R}^{N}}\left(U_{0}^{2}\left(\sum_{i=1}^{k} V_{i}\right) v+U_{0}\left(\sum_{i=1}^{k} V_{i}\right)^{2} u\right) d y,
\end{aligned}
\end{equation}
and $$q^R[u, v] := \frac{1}{2} \left<L_0(u),u\right>_0 + \frac{1}{2} \left<L_1(v),v\right>_1.$$
See \eqref{l0} and \eqref{l1} for the definition of $L_0$, $L_1$. Then $h^R$ is the remainder that is 
\begin{equation}
\begin{aligned}
h^R[u, v]=&-\frac{1}{4}\int_{\mathbb{R}^{N}}\left(4\alpha_0 U_{0} u^{3}+\alpha_0 u^{4}+4\alpha_1\left(\sum_{i=1}^{k} V_{i}\right) v^{3}+ \alpha_1 v^{4}\right) d y \\
&-\frac{\beta}{2} \int_{\mathbb{R}^{N}}\left(2U_{0} u v^{2}+2\left(\sum_{i=1}^{k} V_{i}\right) u^{2} v+u^2v^2\right) d y
\\
&-\frac{\beta}{2} \int_{\mathbb{R}^{N}}\left(U_{0}^{2} v^2+4U_{0}\left(\sum_{i=1}^{k} V_{i}\right) uv+\left(\sum_{i=1}^{k} V_{i}\right)^2u^2\right) d y.
\end{aligned}
\end{equation}
We show in the below that $I\left(U_{0}, \sum_{i=1}^{k} V_{i}\right)$ is the leading order contribution as $k \rightarrow \infty$ and is the only relevant term for the maximization problem.
\begin{lemma} For some $C>0$ independent of $k$
 $$\big|l^R[u_R,v_R]\big| + \big|q^R[u_R,v_R]\big| + \big|h^R[u_R,v_R]\big| \le C k^{-2p} = O\left(k^{-m +1 - 2\delta_0}\right).$$
\end{lemma}
\begin{proof}
 Assertion follows from the estimates \eqref{I2terms}-\eqref{sumdiff}, and \eqref{delta0}.
\end{proof}

Now we expand the main term of the energy functional.
\begin{proposition}\label{A.3} 
\begin{equation*}
I\left(U_0, \left( \sum_{i=1}^k V_{i}\right)\right)  
= A_0+ 
 k\left(A_1+\frac{A_2}{R^{m}}-J(R) e^{-\frac{2R\pi}{k}}\left(\frac{k}{R}\right)^{\frac{N-1}{2}}+O\left(k^{-m-2\delta_0}\right)\right).
\end{equation*}
where   $\displaystyle A_0=\frac{\alpha_0}{4} \int_{\mathbb{R}^{N}} U_0^{4}dy,$ $\displaystyle 
A_1=\frac{\alpha_1}{4} \int_{\mathbb{R}^{N}} V_0^{4}dy,$   $\displaystyle A_2=\frac{a}{2} \int_{\mathbb{R}^{N}} V_0^{2}dy,$ and 
$$0< B_0\le J(R)\le B_1 \quad \text{for some $B_0$ and $B_1$ that are independent of $R$.} $$
\end{proposition}
Before we prove the Proposition \ref{A.3} we prove the following Lemma.
\begin{lemma}\label{lemma_mu} If $R \in S_k$, then 
\[\int_{\mathbb{R}^N}(\mu(|y|)-1) V_1^{2}dy=\frac{a}{R^{m}}\int_{\mathbb{R}^N}V_0^2+O\left(k^{-m-2\delta_0}\right).\]  \end{lemma}
\begin{proof}
For $y \in B_{\frac{R}{2}(x_0)}$ and for large $R$,
\begin{align*}
 \left|\mu(|y+x_1|)-1- \frac{a}{R^{m}}\right| &\le a\left| \frac{1}{|y+x_1|^m} - \frac{1}{R^m}\right| + C\left|\frac{1}{|y+x_1|^{m+\theta}}\right|\\
 &\le C \frac{|y|}{R^{m+1}} +C\left(\frac{1}{R^{m+\theta}}\right).
\end{align*}
Therefore
\begin{equation*}\label{mu1}
\begin{aligned}
&\int_{\mathbb{R}^N}\left(\mu-1\right) V_1^{2}\:dy - \frac{a}{R^{m}}\int_{\mathbb{R}^N}V_0^2 \:dy =  \int_{\mathbb{R}^N}\left(\mu(|y+x_1|)-1- \frac{a}{R^{m}}\right)V_0^2 \:dy\\
&= \int_{B_{\frac{R}{2}(x_0) }}\left(\mu(|y+x_1|)-1- \frac{a}{R^{m}}\right)V_0^2 \:dy + O\left(R^{-m}e^{-\gamma{R}}\right) = O(k^{-m-2\delta_0}).
\end{aligned}
\end{equation*}

\end{proof}

\begin{proof}[proof of Proposition \ref{A.3}] By the definition of $I$ in \eqref{energy} and the equation \eqref{u0v0}
\begin{equation}
\begin{aligned}\label{f1}
&I\left(U_0,  \sum_{i=1}^k V_{i}\right)\\&= \frac{1}{2} \int_{\mathbb{R}^{N}}\left(|\nabla U_0|^{2}+\lambda U_0^{2}-\frac{\alpha_0}{2} U_0^{4}\right)dy
\\&+
 \frac{1}{2} \int_{\mathbb{R}^{N}}\left(\left|\nabla\left( \sum_{i=1}^k V_{i}\right)\right|^{2}+\mu(y) \left( \sum_{i=1}^k V_{i}\right)^{2} -\frac{\alpha_1}{2} \left( \sum_{i=1}^k V_{i}\right)^{4}\right) d y\\
 &-\frac{\beta}{2}\int_{\mathbb{R}^N}U_0^{2} \left( \sum_{i=1}^k V_{i}\right)^{2} d y\\
 &= \frac{\alpha_0}{4} \int_{\mathbb{R}^{N}} U_0^{4}dy + \frac{\alpha_1}{4}\int_{\mathbb{R}^{N}} 2\sum_{j=1}^{k} \sum_{i=1}^{k} V_j^{3} V_i - \left(\sum_{i=1}^k V_i\right)^4\:dy\\
 &+\frac{1}{2} \int_{\mathbb{R}^{N}}\left(\mu-1\right) \left( \sum_{i=1}^k V_{i}\right)^{2}  d y -\frac{\beta}{2}\int_{\mathbb{R}^N}U_0^{2} \left( \sum_{i=1}^k V_{i}\right)^{2} d y.
 \end{aligned}
\end{equation}
We set $\displaystyle A_0:=\frac{\alpha_0}{4} \int_{\mathbb{R}^{N}} U_0^{4}dy$ and we have that $\displaystyle \frac{\beta}{2}\int_{\mathbb{R}^N}U_0^{2} \left( \sum_{i=1}^k V_{i}\right)^{2} d y \le C e^{-\gamma R}.$ Nextly,
\begin{equation}
\begin{aligned}\label{f3}
\int_{\mathbb{R}^{N}}(\mu-1) \left( \sum_{i=1}^k V_{i}\right)^{2}dy&=\int_{\mathbb{R}^{N}}(\mu-1) \sum_{i=1}^k V_{i}^2 \:dy +  \int_{\mathbb{R}^N}(\mu-1)  \sum_{i\ne j} V_iV_j\:dy
\end{aligned}
\end{equation}
and  
\begin{align*}
J_1:=&\int_{\mathbb{R}^N}(\mu-1)  \sum_{i\ne j} V_iV_j\:dy \le Ck\int_{\Omega_{1}}(\mu(|y|)-1)  \left(\sum_{j\ge 2}V_j\right)\left(\sum_{i=1}^k V_i\right)\:dy\\
&\le Ck e^{-\frac{\eta R \pi}{k}} \int_{\mathbb{R}^N} (\mu-1) e^{-\gamma'|y-x_1|} \:dy, \quad \text{for some $\eta \in (0,1)$ and $\gamma'>0$}.
\end{align*}
For the last integral, the proof of Lemma \ref{lemma_mu} for $e^{-\gamma'|y-x_1|}$ in place of $V_1(y)^2$ gives that $J_1 = O\left (k R^{-m} e^{-\frac{\eta R \pi}{k}}\right)$ and $\displaystyle\int_{\mathbb{R}^{N}}(\mu-1) \sum_{i=1}^k V_{i}^2 \:dy = k\int_{\mathbb{R}^{N}}(\mu-1) V_1^2 \:dy$.
Now,
\begin{align*}
&\int_{\mathbb{R}^{N}} 2\sum_{j=1}^{k} \sum_{i=1}^{k} V_j^{3} V_i - \left(\sum_{i=1}^k V_i\right)^4\:dy\\
&=\int_{\mathbb{R}^{N}} \sum_{i=1}^k V_i^4 - 2 \left(V_1^3 \sum_{j\ne 1} V_j + \cdots + V_k^3 \sum_{j\ne k} V_j \right) - \sum_{\substack{(i,j,k,\ell)\\ \text{at most 2} \\ \text{of indices are same}} }V_iV_j V_k V_\ell \:dy.
\end{align*}
In view of  the symmetry
$$ \int_{\mathbb{R}^{N}} \sum_{i=1}^k V_i^4 - 2 \left(V_1^3 \sum_{j\ne 1} V_j + \cdots + V_k^3 \sum_{j\ne k} V_j \right) \:dy =  k\int_{\mathbb{R}^{N}} V_0^4  -2  V_1^3 \sum_{j\ge 2} V_j \:dy.$$
%
Also, since
\begin{align*}
&\{ (i,j,k,\ell)~|~ \text{at most 2 of indices are same}\} \subset \{ (i,j,k,\ell)~|~ \text{at least 2 of indices are not $1$} \}, \\
&J_2:=\int_{\mathbb{R}^N} \sum_{\substack{(i,j,k,\ell)\\ \text{at least 3} \\ \text{of indices distinct}} }V_iV_j V_k V_\ell\:dy \le Ck \int_{\Omega_1}  \left( \sum_{j\ge 2}^k V_j \right)^2 \left(\sum_{i=1}^k V_i\right)^2 \:dy.
\end{align*}
%
%
By the same estimate in \eqref{sumdiff}, $J_2 = O(k^{-2p}) = O(k^{-m+1-2\delta_0})$. In summary,  we have
\begin{equation}
\begin{aligned}\label{f5}
&I\left(U_0, \left( \sum_{i=1}^k V_{i}\right)\right)  
\\& = A_0+ 
 k\left(A_1+\frac{A_2}{R^{m}}-\frac{\alpha_1}{2} \left( \sum_{i=2}^{k} \int_{\mathbb{R}^{N}} V_1^{3} V_i\:dy\right) + O\left(k^{-m-2\delta_0}\right)\right) .
\end{aligned}
\end{equation}
Finally,
\begin{align*}
&\sum_{i=2}^{k} \int_{\mathbb{R}^{N}} V_1^{3} V_i\:dy\\
& = \sum_{i=2}^{k} \left(e^{-\left|x_{1}-x_{i}\right|}|x_1-x_i|^{-\frac{N-1}{2}}\right) \int_{\mathbb{R}^{N}} V_1(y)^{3} V_i(y) \left(e^{\left|x_{1}-x_{i}\right|}|x_1-x_i|^{\frac{N-1}{2}}\right) \:dy\\
&= J_3\sum_{i=2}^{k} \left(e^{-\left|x_{1}-x_{i}\right|}|x_1-x_i|^{-\frac{N-1}{2}}\right),\\
 \text{where} \quad J_3 &=  \int_{\mathbb{R}^{N}} V_0(y)^{3} V_0(y-x_i+x_1) \left(e^{\left|x_{1}-x_{i}\right|}|x_1-x_i|^{\frac{N-1}{2}}\right) \:dy.
\end{align*}
For sufficiently large $R$ and for $y$ near origin, $V_0(y-x_i+x_1)e^{|y-x_i+x_1|} |y-x_i+x_1|^{\left(\frac{N-1}{2}\right)}$ is positive and away from $0$. Since $  -|y| \le -|y-x_i+x_1| + |x_i-x_1| \le |y|$, we have that
$$e^{-|y-x_i+x_1|+|x_i-x_1|} \left(\frac{|x_i-x_1|}{|y-x_i+x_1|}\right)^{\left(\frac{N-1}{2}\right)} \ge \frac{1}{2} e^{-|y|}$$
for $|y|<1$ and $R$ sufficiently large, and thus $\displaystyle J_3 \ge \int_{B_1(0)} V_0(y)^3\left( \frac{1}{2}e^{-|y|}\right) \:dy.$ Also, since 
\begin{align*}
 &V_0(y)V_0(y-x_i+x_1) \left(e^{\left|x_{1}-x_{i}\right|}|x_1-x_i|^{\frac{N-1}{2}}\right) \\
 &\le C e^{-|y| - |y-x_i+x_1| + |x_i-x_1|}\left(\frac{|x_i-x_1|^{\left(\frac{N-1}{2}\right)}}{(1+|y|^{\left(\frac{N-1}{2}\right)})\cdot(1+|y-x_i+x_1|^{\left(\frac{N-1}{2}\right)})}\right) \le C,
\end{align*}
we have that $\displaystyle J_3 \le C\int_{\mathbb{R}^N} V_0(y)^2 \:dy$.
\end{proof}
Now we are ready to prove Theorem \ref{mainthm}.  
\begin{proof}[Proof of Theorem \ref{mainthm}]
By Proposition \ref{A.3} 
$$
F(R)= A_0+ 
 k\left(A_1+\frac{A_2}{R^{m}}-J(R)e^{-\frac{2R\pi}{k}}\left(\frac{k}{R}\right)^{\frac{N-1}{2}}+o\left(\frac{1}{k^{m+2\delta_0}}\right)\right).
$$Since  the function $g(R):=\frac{A_2}{R^{m}}-J(R)e^{-\frac{2R\pi}{k}}\left(\frac{k}{R}\right)^{\frac{N-1}{2}}$ has a maximum point in the interior of $S_k$ for $k$ sufficiently large,    $\max _{R \in S_{k}} F(R)$ is achieved by an interior point $R_0$ of $S_{k}$. Therefore,  we can conclude that  $(U_{R_0}, V_{R_0})$ is a solution of \eqref{system},  completing the proof of Theorem \ref{mainthm}.
\end{proof}

\bibliographystyle{amsplain}

\end{document}